\documentclass[a4paper,11pt]{amsart}

\usepackage{amsmath,amssymb,amsthm,amsfonts,bbm}
\usepackage{color}
\usepackage{tikz}
\usetikzlibrary{cd}
\usepackage{hyperref}
\usepackage{stmaryrd} % llbracket rrbracket
\usepackage{enumitem}
\usepackage{multicol}
\usepackage{graphicx}
\usepackage{mathtools}
\usepackage{dutchcal} %for sheaf-hom
\setitemize[0]{leftmargin=*}
\setenumerate[0]{leftmargin=*}
\let\OLDenumerate\enumerate
\renewcommand\enumerate{\OLDenumerate\addtolength{\itemsep}{0.5ex}}
\definecolor{rosso}{RGB}{162,0,0}
\definecolor{verde}{RGB}{0,100,0}
\definecolor{blu}{RGB}{0,0,162}
\definecolor{bla}{RGB}{100,0,100}

\newcommand{\domark}{%
  \vbox to 0pt{
    \kern-\dp\strutbox
    \smash{\llap{*\kern1em}}
    \vss
  }%
}

\usepackage[normalem]{ulem}

\newcommand{\OBSOLETE}[1]{}

\newcommand{\MD}{\textsf{MD}}
\newcommand{\kk}{\mathbbm{k}} % field
\newcommand{\QQ}{\mathbb{Q}}
\newcommand{\ZZ}{\mathbb{Z}}
\newcommand{\NN}{\mathbb{N}}

\newcommand{\PP}{\mathbb{P}} % projective space
\newcommand{\stX}{\mathcal{X}} % stack
\newcommand{\stY}{\mathcal{Y}}
\newcommand{\stU}{\mathcal{U}} % open substack
\newcommand{\stV}{\mathcal{V}}
\newcommand{\stZ}{\mathcal{Z}}
\newcommand{\BB}{\mathcal{B}} % classifing space

\newcommand{\can}{{\mathop{can}}}
\newcommand{\Xcan}{\stX^{\can}}

\newcommand{\stackquot}[2]{\left[ \raisebox{0.5ex}{$\displaystyle #1$} \middle/ \raisebox{-0.5ex}{$\displaystyle #2$} \right]}

\newcommand{\torP}{\mathcal{P}} % torsor
\newcommand{\torQ}{\mathcal{Q}}

\newcommand{\gsR}{\mathcal{R}} % graded sheaves (of algebras)
\newcommand{\gsS}{\mathcal{S}}
\newcommand{\gsI}{\mathcal{I}} % ideal in \gsS

\newcommand{\famF}{\mathcal{F}} % family of sheaves
\newcommand{\famG}{\mathcal{G}}

\newcommand{\shL}{\mathcal{L}} % line bundle or reflexive sheaf
\newcommand{\shF}{\mathcal{F}} % coherent sheaf
\newcommand{\shG}{\mathcal{G}} % coherent sheaf
\newcommand{\OO}{\mathcal{O}} % structure sheaf

\newcommand{\agG}{\mathsf{G}} % affine group scheme
\newcommand{\agH}{\mathsf{H}} 
\newcommand{\D}[1]{\mathsf{D}(#1)}
\newcommand{\Gm}{\mathbb{G}_m} % multiplicative group

\DeclareMathOperator{\Spec}{Spec}
\let\div\relax
\DeclareMathOperator{\div}{div}

\DeclareMathOperator{\Pic}{Pic}
\DeclareMathOperator{\PPic}{\overline{\Pic}}
\DeclareMathOperator{\Refl}{Ref}
\DeclareMathOperator{\RRef}{\overline{\Refl}}
\DeclareMathOperator{\WDiv}{WDiv}
\DeclareMathOperator{\Cl}{Cl}

\newcommand{\RS}{\mathcal{R}} % Cox sheaf
\newcommand{\RR}{\mathsf{R}} % Cox ring
\newcommand{\PicRS}{\RS_{\Pic}}
\newcommand{\PicRR}{\RR_{\Pic}}
\newcommand{\RefRS}{\RS_{\Refl}}
\newcommand{\RefRR}{\RR_{\Refl}}
\renewcommand{\SS}{\mathcal{S}} % R = S/I
\newcommand{\SR}{\mathsf{S}} % H^0(\SS)

\DeclareMathOperator{\Sch}{Sch}
\DeclareMathOperator{\AbGrps}{AbGrps}
\DeclareMathOperator{\Grps}{Grps}
\DeclareMathOperator{\Algs}{Algs}
\DeclareMathOperator{\QCoh}{QCoh}

\DeclareMathOperator{\Alg}{AlgQCoh}
\mathchardef\mhyphen="2D
\newcommand{\mhyph}{\mathbin{\mhyphen}}
\DeclareMathOperator{\GrAlg}{Gr-AlgQCoh}
\DeclareMathOperator{\Aff}{Aff}

\DeclareMathOperator{\Hom}{Hom}
\DeclareMathOperator{\SHom}{\mathcal{Hom}} % SheafHom
\DeclareMathOperator{\Ext}{Ext}
\DeclareMathOperator{\id}{id}

\newcommand{\twovee}{^{\vee\vee}} % double dual
\newcommand{\coloneqq}{\mathrel{\mathop:}=}
\newcommand{\eqqcolon}{=\mathrel{\mathop:}}
\newcommand{\xto}{\xrightarrow}
\newcommand{\into}{\hookrightarrow}
\newcommand{\onto}{\twoheadrightarrow}
\newcommand{\isom}{\xrightarrow{\raisebox{-0.5ex}{$\scriptstyle\sim$}}}

\newcommand{\blank}{\,\text{--}\,} %blank spot to define functors

\theoremstyle{plain}
\newtheorem{theorem}{Theorem}[section]
\newtheorem*{theorem*}{Theorem}
\newtheorem*{lemma*}{Lemma}

\theoremstyle{definition}
\newtheorem*{example*}{Example}

\numberwithin{equation}{section}

\newcommand{\bib}[5]{{\bibitem{#1} #2, {\emph{#3},} #4#5.}}
\newcommand{\arXiv}[1]{{\href{http://arxiv.org/abs/#1}{\texttt{arXiv:#1}}}}

\hypersetup{
    colorlinks,
    linkcolor={red!50!black},
    citecolor={blue!50!black},
    urlcolor={blue!80!black}
}
\usepackage{aliascnt} % to get hyperref's \autoref to work properly with counters

\theoremstyle{plain}

  \newtheorem{maintheorem}{Theorem}
  \newaliascnt{proposition}{theorem}
  \newtheorem{proposition}[proposition]{Proposition}
  \aliascntresetthe{proposition}

  \newaliascnt{lemma}{theorem}
  \newtheorem{lemma}[lemma]{Lemma}
  \aliascntresetthe{lemma}

  \newaliascnt{corollary}{theorem}
  \newtheorem{corollary}[corollary]{Corollary}
  \aliascntresetthe{corollary}

\theoremstyle{definition}

  \newaliascnt{definition}{theorem}
  \newtheorem{definition}[definition]{Definition}
  \aliascntresetthe{definition}

  \newaliascnt{remark}{theorem}
  \newtheorem{remark}[remark]{Remark}
  \aliascntresetthe{remark}

  \newaliascnt{example}{theorem}
  \newtheorem{example}[example]{Example}
  \aliascntresetthe{example}

\usepackage{todonotes}

\title[Cox rings of stacks]{Cox rings of algebraic stacks}

\author[A. Hochenegger]{Andreas Hochenegger}
\address{Dipartimento di Matematica ``Francesco Brioschi'', Politecnico di Milano, via Bonardi 9, 20133 Milano, Italy}
\email{andreas.hochenegger@polimi.it}
\author[E. Martinengo]{Elena Martinengo}
\address{Dipartimento di Matematica ``Giuseppe Peano'', Universit\`a degli Studi di Torino, via Carlo Alberto 10, 10123 Torino, Italy}
\email{elena.martinengo@unito.it}
\author[F. Tonini]{Fabio Tonini}
\address{Dipartimento di Matematica e Informatica ``Ulisse Dini'', Universit\`a degli Studi di Firenze, Viale Morgagni 67, 50134 Firenze, Italy}
\email{jacobbb84@gmail.com}

\begin{document}

\maketitle

\begin{abstract}
We give a proper definition of the multiplicative structure of the following rings: the Cox ring of invertible sheaves on a general algebraic stack; and the Cox ring of rank one reflexive sheaves on a normal and excellent algebraic stack.

\noindent
We show that such Cox rings always exist and establish their \mbox{(non-)}\allowbreak unique\-ness in terms of an $\Ext$-group.
Moreover, we compare our definition with the classical construction of a Cox ring on a variety.
Finally, we give an application to the theory of Mori dream stacks.
\end{abstract}

\section{Introduction}

Let $\stX$ be an algebraic stack, e.g. a variety over some field.
The \emph{Cox ring of line bundles} on $\stX$ as an $H^0(\stX,\OO)$-module is 
\[
\PicRR(\stX) =\ \bigoplus_{\mathclap{\shL \in \Pic(\stX)}}\; H^0(\stX,\shL).
\] 
As in the case of varieties, the question about the multiplicative structure is delicate, and in the articles \cite{HM1,HM2} the first and second named authors glossed over this technicality
when speaking about the Cox ring of algebraic stacks.

The purpose of the paper is to give a very general definition of a Cox ring, that is a multiplicative structure on $\PicRR(\stX)$. 
%Using the language of families of sheaves and the theory of torsors, we can show that a Cox ring of line bundles does always exist. Moreover, we discuss the uniqueness of a Cox ring. 
The following main theorem summarises our results; see \autoref{thm:mainthm-pic}.

\begin{maintheorem}\label{thm:main theorem line bundles}
Let $\stX$ be an algebraic stack. 
Then a Cox ring of line bundles on $\stX$, also called a $\Pic$-Cox ring $\PicRR(\stX)$, exists.
Moreover, a $\Pic$-Cox ring is unique up to isomorphism if
\[
\Ext^1(\Pic(\stX),H^0(\OO_\stX^*)) = 0,
\]
which, in particular, holds in the cases that
\begin{itemize} 
\item $\Pic(\stX)$ is free; or 
\item $H^0(\OO_\stX^*) = \kk^*$ for an algebraically closed field $\kk$.
\end{itemize}
\end{maintheorem}

In the case that $\stX$ is a noetherian, normal and excellent algebraic stack, 
it makes sense to speak about the Cox ring of reflexive sheaves of rank one. 
We denote the group of isomorphism classes of such sheaves by $\Refl_1(\stX)$. We obtain an analogous statement for the Cox ring of such sheaves; see \autoref{prop:picref}.

\begin{maintheorem}
If $\stX$ is a noetherian, normal and excellent algebraic stack then a  Cox ring of reflexive sheaves of rank one
\[
\RefRR(X) = \bigoplus_{\mathclap{\shL \in \Refl_1(\stX)}}\; H^0(\stX,\shL),
\] 
also called a $\Refl_1$-Cox ring, exists.
Moreover, a $\Refl_1$-Cox ring is unique up to isomorphism if
\[
\Ext^1(\Refl_1(\stX),H^0(\OO_\stX^*)) = 0,
\]
which, in particular, holds in the cases that
\begin{itemize} 
\item $\Refl_1(\stX)$ is free; or 
\item $H^0(\OO_\stX^*) = \kk^*$ for an algebraically closed field $\kk$.
\end{itemize}
\end{maintheorem}

Under these conditions on $\stX$, we show this theorem by restricting to the regular locus $\stU \subseteq \stX$, which induces an isomorphism of $\Refl_1$-Cox rings of $\stX$ and $\Pic$-Cox rings of $\stU$.

We note that our construction does not need the Picard group (or the group of reflexive sheaves of rank one) to be finitely generated, although also the classical construction in \cite{Coxrings} does not really rely on it.
Additionally the classical construction assumes that $H^0(\OO_\stX^*) = \kk^*$ for an algebraically closed field $\kk$. We see here clearly that this becomes important for the uniqueness of the Cox ring, and not for its existence.

Another point to remark is that no kind of regularity is needed on $\stX$ to state and prove the results about the $\Pic$-Cox ring. Actually in this paper we will not even assume that $\stX$ is an algebraic stack, just a category fibered in groupoids admitting an fpqc atlas, called a \emph{pseudo-algebraic fibered category} (see \autoref{def:pseudo-algebraic}). 

There is an interesting case where we can use this greater generality of the construction developed here, namely in the case of the infinite root stack associated to a logarithmic scheme (see \autoref{ex:logarithm} for details).
A priori it is not clear what kind of information about the logarithmic structure can be obtained from its Cox ring, but we plan to investigate this further.

The main ingredient in the proof of the above results and in the study of multiplicative structures on $\RR(X)$ is the notion of a \emph{family of sheaves}: if $G$ is an abelian group then a $G$-family of sheaves is a collection of quasi-coherent sheaves $(\famF_g)_{g\in G}$ together with morphisms 
\[
\OO_\stX\to \famF_0 \text{ and } \famF_g\otimes \famF_{g'}\to \famF_{g+g'}
\]
satisfying certain compatibility conditions (see \autoref{def:family-sheaves}). With a $G$-family $\famF$, one can associate
\[
\RS_\famF=\bigoplus_g \famF_g \text{ and } \RR_\famF=\bigoplus_g H^0(\stX,\famF_g)
\]
which we call the \emph{Cox sheaf} and \emph{Cox ring} of $\famF$ and which are a sheaf of $\OO_\stX$-algebras and a $H^0(\OO_\stX)$-algebra, respectively.

The Ext-groups in the two theorems above parametrize the possible ways to put a structure of a family of sheaves on the collections $(\shL)_{\shL\in \Pic(\stX)}$ and $(\shF)_{\shF\in \Refl_1(\stX)}$, respectively.

In particular, if $\stX$ is a quasi-compact algebraic stack, there exists a $\Pic$-Cox sheaf of algebras 
\[
\PicRS(\stX)=\bigoplus_{\shL\in \Pic(\stX)}\shL
\]
such that $H^0(\PicRS(\stX))=\PicRR(\stX)$. Moreover, if $\stX$ is a Noetherian, normal and excellent algebraic stack, there exists a $\Refl_1$-Cox sheaf of algebras 
\[
\RefRS(\stX)=\bigoplus_{\shF\in \Refl_1(\stX)}\shF
\]
such that $H^0(\RefRS(\stX))=\RefRR(\stX)$.

The crucial point is that a $G$-family of line bundles can be thought of as a torsor under the dual diagonalisable group scheme $\D G$. All the results are then consequence of constructions and known results about torsors.

\medskip 
Finally, as an application, we obtain the following result for a Mori dream space (see \autoref{thm:piccan} and \autoref{rem:generaliseHM2}):

\begin{maintheorem}
Let $X$ be a Mori dream space over an algebraically closed field $\kk$ of characteristic zero.
%Denote by $\RefRR(X)$ and $\RefRS(X)$ the $\Refl_1$-Cox ring and Cox sheaf, respectively. 
Then 
\[
\Xcan=\stackquot{\Spec_X \RefRS(X)}{\D{\Cl(X)}}.
\]
is a Mori dream quotient stack in the sense of \cite{HM2}, that is,
\[
\Xcan=\stackquot{\Spec \PicRR(\Xcan) \setminus V}{\D{\Pic(\Xcan)}},
\]
where %$\RR(\Xcan)$ is the $\Pic$-Cox ring of $\Xcan$ and 
$V$ is locus defined by the \emph{irrelevant ideal}.

Moreover, let $\pi \colon \Xcan \to X$ the structure map. Then
\[
\Refl_1(\Xcan)=\Pic(\Xcan)=\Cl(X) \text{ and } \pi_*\PicRS(\Xcan) \simeq \RefRS(X)
\]
so that $\PicRR(\Xcan)=\RefRR(X)$. %, where $\PicRS(\Xcan)$ is the $\Pic$-Cox sheaf of $\Xcan$.
\end{maintheorem}

This generalises \cite[Prop. 2.9]{HM2} where we used the additional assumption that $H^0(\Spec \RefRR(X),\OO^*) = \kk^*$ and $\Pic(\Spec \RefRR(X)) = 0$. We point out that, in \cite{HM1} and \cite{HM2}, $\Xcan$ is the starting point to build Mori dream stacks by so-called root constructions.
Consequently, this assumption can be dropped in further results there, see \autoref{rem:generaliseHM2}.

In \cite{Derenthal-etal}, the authors consider a different notion of Cox rings for varieties over a field $\kk$. 
The main difference is that this is a relative notion, in the sense that it depends on the base field $\kk$, while our definition is absolute. Starting from $X/\kk$ and denoting by $G$ the absolute Galois group of $\kk$, they consider $G$-equivariant maps $\lambda\colon M\to \Pic(\overline X)$, where $M$ is a $G$-module and $\overline X$ is the base change of $X$ to a separable closure of $\kk$, and Cox sheaves $\RS$ of type $\lambda$ (see \cite[Defs. 2.2 \& 3.1]{Derenthal-etal}). 
When $\kk$ is separably closed, such an object coincides with an $M$-family of line bundles over $X$. Compare also \cite[Thm. 1.1]{Derenthal-etal} and \autoref{rem:torsor-families}.  We expect that several constructions and results in  \cite{Derenthal-etal} can be extended to stacks and it is our plan to investigate this further.

In a similar, relative setting, a classification of torsors under groups of multiplicative type analogous to \autoref{thm:main theorem line bundles} via \autoref{rem:torsor-families} is given in \cite[Thm. 1.5.1]{CTS}.

The paper is divided as follows. In \autoref{sec:families-torsors}, we give a very general definition of a Cox ring, introducing the language of $G$-families of sheaves, where $G$ is an abelian group.
In the case of families of line bundles, this notion is equivalent to torsors, see \autoref{sec:torsors}. 
\autoref{lem:induced-family} is the key lemma of our construction: it gives a criterion for a $G$-family of line bundles to induce a family of a quotient $G/H$.
We apply this proposition in \autoref{sec:coxrings}, where we first show the existence of a Cox ring of line bundles using a free resolution of the Picard group.
Moreover, we measure the (non-)uniqueness of Cox rings with an $\Ext$-group.
This is the content of \autoref{thm:families-extension}.
In the case that it makes sense to speak about reflexive sheaves of rank one, 
we obtain the same result, see \autoref{sec:refcox}.
In \autoref{sec:confrontation}, we compare our construction with the classical one given in \cite{Coxrings}. 
The upshot is that our key lemma is already present there.
Finally, in \autoref{sec:application} we give the application to Mori dream spaces.

\subsection*{Acknowledgements}
We thank Ariyan Javanpeykar and Ronan Terpereau for the discussions together at the workshop \emph{Stacks in Turin} at the University of Turin, September 2017, 
which were a reason for writing this article.
We also thank Marta Pieropan for making us aware of the articles \cite{CTS, Derenthal-etal}.
Finally we thank Angelo Vistoli for pointing out \autoref{ex:nonunique:unique} to us.

\section{Families of sheaves, torsors and their Cox rings}
\label{sec:families-torsors}

\begin{definition}\cite[Definition 1.2]{ToniniSheaf}
\label{def:pseudo-algebraic}
A category fibered in groupoids $\stX$ is called \emph{pseudo-algebraic} if it admits an fpqc atlas $U\to \stX$ from a scheme $U$, that is a map such that, for all maps $V\to \stX$ from a scheme, the base change $V\times_\stX U\to V$ is an fpqc covering of algebraic spaces.

Moreover $\stX$ is called \emph{quasi-compact} if it admits an fpqc atlas from a quasi-compact scheme. 
\end{definition}

For pseudo-algebraic fibered categories, the notions of quasi-coherent sheaves, vector bundles, line bundles, their pullbacks and pushforwards are well defined and well behaved. 
For an introduction to the subject, we refer to \cite[\S 4.1]{TalpoVistoli} or \cite[\S 1]{ToniniSheaf}.

The idea is that a quasi-coherent sheaf on a pseudo-algebraic fibered category $\stX$ assigns in a compatible way to any object of $\stX$ over a scheme $S$ a quasi-coherent sheaf on $S$. Therefore $\stX$ can be thought of as a sort of parameter space for quasi-coherent sheaves.

Note that this notion of sheaves extends the classical one for schemes and algebraic stacks. The pseudo-algebraic condition guarantees that the category $\QCoh(\stX)$ of quasi-coherent sheaves is an abelian category (see \cite[Proposition 1.4]{ToniniSheaf} and other nice properties.

\medskip
In this section we consider a pseudo-algebraic fibered category $\stX$, for example a variety or an algebraic stack.
We describe a general procedure that associates with a family of sheaves on $\stX$ with certain properties an $H^0(\OO_\stX)$-algebra of sections of them, that we call the \emph{Cox ring} associated with the family.
Moreover, we show existence of Cox rings of line bundles using the language of torsors and we discuss their uniqueness. Finally, we extend these results to Cox rings of reflexive sheaves of rank 1.

\subsection{Families of sheaves}
\label{sec:families}

\begin{definition}
\label{def:family-sheaves}
Let $G$ be an abelian group. A \emph{$G$-family of sheaves} $\famF$ on $\stX$ is a collection of quasi-coherent and finitely presented sheaves $\famF_g$ on $\stX$ for $g \in G$ together with
\begin{itemize}
\item an isomorphism $\xi_0\colon \OO_\stX \to \famF_0$ and
\item bilinear morphisms $\xi_{g,g'}\colon \famF_g \times \famF_{g'} \to \famF_{g+g'}$ for  all $g, g'\in G$,
such that the morphisms
\[ 
\omega_{g,g'}\colon \famF_g \to \SHom_{\stX}(\famF_{g'}, \famF_{g+g'}), \quad x \mapsto \xi_{g,g'}(x,\blank)
\]
are isomorphisms,
\end{itemize}
We require that the following diagrams are commutative for all $g,g',g''\in G$:
\begin{itemize}
\item ``associativity'':
\[
\begin{tikzcd}
\famF_g\times \famF_{g'}\times \famF_{g''}  \ar[r, "\xi_{g,g'}\times \id"] \ar[d, "\id \times \xi_{g',g''}"'] & \famF_{g+g'}\times \famF_{g''} \ar[d, "\xi_{g+g',g''}"] \\
\famF_g\times \famF_{g'+g''}  \ar[r, "\xi_{g,g'+g''}"] & \famF_{g+g'+g''} 
\end{tikzcd}
\]
\item ``commutativity'':
\[
\begin{tikzcd}
\famF_g\times \famF_{g'} \ar[r, "\text{swap}"] \ar[d, "\xi_{g,g'}"'] & \famF_{g'}\times \famF_{g} \ar[d, "\xi_{g',g}"] \\
\famF_{g+g'} \ar[r, "\id"] & \famF_{g'+g}  
\end{tikzcd}
\]
\item ``unit'':
\[
\begin{tikzcd}[column sep=small]
\famF_g \ar[rr, "\id"] \ar[dr, "\id\times\xi_0"'] && \famF_g \\
& \famF_g \times \famF_0 \ar[ur, "\xi_{g,0}"']\\
\end{tikzcd}
\]
\end{itemize}
We denote a $G$-family by $\famF=(\famF,\xi)$ or by $\famF\colon G\to \QCoh(\stX)$.
\end{definition}

\begin{remark}
In \autoref{def:family-sheaves} one can also directly use morphisms $\xi_{g,g'}\colon \famF_g \otimes \famF_g'\to \famF_{g+g'}$ and change the commutative diagrams accordingly.
\end{remark}

\begin{remark}
\label{rem:reflexive-torsor}
Let $\famF$ be a $G$-family of sheaves on $\stX$. 
Note that
\[
\begin{tikzcd}
\famF_g \ar[r, "\omega_{g,-g}"] & \SHom(\famF_{-g},\famF_0) \ar[rr, "{\SHom(-,\xi_0^{-1})}"] && \SHom(\famF_{-g},\OO_\stX) = \famF_{-g}^\vee
\end{tikzcd}
\]
is an isomorphism and the composition $\famF_g \to \famF_{-g}^\vee \to \famF_g\twovee$ is the bidual map.
Thus, all the sheaves $\famF_g$ are reflexive.

If $\famF_h$ is invertible for some $h\in G$ then the maps $\xi_{g,h}$ induce isomorphisms $\famF_g \otimes \famF_h \to \famF_{g+h}$ for all $g\in G$.
\end{remark}

\begin{definition}
Let $\famF$ be a $G$-family and $\famG$ an $H$-family together with a group homomorphism $\alpha \colon G \to H$. 
A \emph{morphism of families $\phi \colon \famF\to\famG$ with respect to $\alpha$} consists of $\{\phi_g \colon \famF_g \to \famG_{\alpha(g)}\}_{g\in G}$ such that the following diagrams commute for all $g,g' \in G$:
\[
\begin{tikzcd}[column sep=small]
& \OO_{\stX} \ar[dl, "\xi^{\famF}_0"'] \ar[dr, "\xi^{\famG}_0"] \\
\famF_0 \ar[rr, "\phi_0"] && \famG_0
\end{tikzcd}
\quad
\begin{tikzcd}
\famF_g \times \famF_{g'} \ar[r, "\phi_g \times \phi_{g'}"] \ar[d, "\xi^{\famF}_{g,g'}"'] & \famG_{\alpha(g)} \times \famG_{\alpha(g')} \ar[d, "\xi^{\famG}_{\alpha(g),\alpha(g')}"]\\
\famF_{g+g'} \ar[r, "\phi_{g+g'}"] & \famG_{\alpha(g+g')}
\end{tikzcd}
\]
One can check that such a $\phi\colon \famF \to \famG$ is also compatible with the remaining three diagrams in the definition of a $G$-family.

We call a morphism of families $\phi \colon \famF \to \famG$ with respect to $\alpha=\id_G$ just a \emph{morphism of $G$-families}.
\end{definition}

The definition above allows us to speak of the category of families on a given space $\stX$.

\begin{definition}
Let $G$ be an abelian group and $\famF$ be a $G$-family of sheaves. The $H^0(\OO_\stX)$-module 
\[ 
\RR_\famF \coloneqq \bigoplus_{g\in G} H^0(\famF_g)
\]
is called the \emph{Cox ring} associated with $\famF$.
The ring structure is given by
\[ 
H^0(\famF_g) \times H^0(\famF_{g'}) \xto{H^0(\xi_{g,g'})} H^0(\famF_{g+g'}) 
\]
which turns $\RR_\famF$ into a commutative $G$-graded $H^0(\OO_\stX)$-algebra with structure morphism  $H^0(\OO_\stX)\xto{\xi_0} H^0(\famF_0).$
\end{definition}

\begin{definition}
\label{def:cox sheaf}
We define the \emph{Cox sheaf of algebras} as
\[
\RS_\famF = \bigoplus_{g\in G} \famF_g
\]
using the $G$-family structure. 
\end{definition}

\begin{remark}
There is always a morphism of $H^0(\OO_\stX)$-algebras
\[
\RR_\famF\to H^0(\RS_\famF)
\]
and it is an isomorphism if $G$ is finite or $\stX$ is quasi-compact, that is it has an fpqc atlas $U\to \stX$ from a quasi-compact scheme $U$.
\end{remark}

\begin{remark}
\label{rem:pullback-families}
If $f\colon\stY \to \stX$  is a morphism of fibered categories then 
$f^*\famF = (f^*\famF_g, f^*\xi_0,f^*\xi_{g,g'})_{g,g'\in G}$ is in general not a $G$-family: 
the problem is that $f^*\famF_u\to \SHom_\stY(f^*\famF_v, f^*\famF_{u+v})$ may fail to be an isomorphism. 

If all the sheaves involved are invertible or the map $f\colon\stY \to \stX$ is flat, then $f^*\famF$ is a $G$-family of sheaves on $\stY$ and we have a $G$-graded $H^0(\OO_\stX)$-linear map $\RR_\famF \to \RR_{f^*\famF}$.
\end{remark}

\begin{remark}
The assignment $\famF \mapsto \RR_\famF$ mapping a family $\famF$ to its Cox ring is functorial, i.e.\ given a morphism of families $\phi \colon \famF \to \famG$ with respect to a group homomorphism $\alpha$, there is an induced ring homomorphism $\RR_\phi \colon \RR_\famF \to \RR_\famG$.
\end{remark}

\subsection{Generalities on torsors}
\label{sec:torsors}

We recall here the classical definition of torsor over a scheme and the (less classical) definition of torsors over a fibered category.

\begin{definition}
Let $R$ be a ring and $\agG$ be an affine group scheme over $R$  
such that $\agG \to \Spec R$ is flat (hence faithfully flat). Let $X$ be a scheme over $R$.
A \emph{$\agG$-torsor} over $X$ is a faithfully flat morphism $P \to X$
with an action of $\agG$ on $P$ as a scheme over $X$ such that it is locally trivial in the fpqc-topology.

We denote by with $\BB_R \agG$, or just $\BB \agG$ for short, the stack of $\agG$-torsors over $R$.
\end{definition}

\begin{remark}
The local triviality means that  there exists a fpqc-covering $\{U\to X\}_{U \in \mathfrak U}$ of $X$ and a $\agG$-equivariant isomorphism
\[
P_{|U} \simeq \agG \times U, \quad \forall U \in \mathfrak U.\]
It is equivalent to $\agG \times P \to P \times_X P, (h,p) \mapsto (hp,p)$ being an isomorphism.
\end{remark}

\begin{definition}
Let $\agG$ be a flat and affine group scheme over $R$.
Let $\stX$ be a category fibered in groupoids over $R$.
A \emph{$\agG$-torsor} over $\stX$ is a morphism $\torP \to \stX$ of categories fibered in groupoids over $R$ such that, for all morphisms $X \to \stX$ from a scheme, the base change $\torP\times_\stX X\to X$ is a $\agG$-torsor over $X$ and all those torsors are compatible in the obvious way. 
\end{definition}

Notice that, since we are restricting ourselves to affine group schemes, torsors are affine maps.

\begin{remark}
A $\agG$-torsor over $\stX$ is the same as a morphism of fibered categories $\stX\to \BB_R \agG$. In this case $\torP \to \stX$ is the base change along the trivial torsor $\Spec R\to \BB_R \agG$.
\end{remark}

\begin{remark}
If $\stX$ is an algebraic stack over $R$ (resp. a scheme, a pseudo-algebraic fibered category) then so is the total space of a $\agG$-torsor.
\end{remark}

\begin{remark}
\label{rem:quotients}
Let $\agG$ be a flat an affine group scheme over $R$.
If $\stY\to \stX$ is a representable map of fibered categories over $R$, e.g. an affine map, an action of $\agG$ on $\stY$ over $\stX$ consists in the following data: for all objects  $T\to \stX$ an action of $\agG$ (over $T$) on the algebraic space $\stY\times_\stX T$ which are compatible in an obvious way.
Here, considering only actions relative to $\stX$, we want to avoid to talk about action of groups on fibered categories in general.

If $\stY\to \stX$ is a representable map with an action of $\agG$, then we can form the quotient ``stack'' (the stackiness is only relative to the base) $[\stY/\agG]\to \stX$.
Its central property is that
for an object $T\to \stX$ and $\stY_T=\stY\times_\stX T$ we have
\[
[\stY/\agG] \times_\stX T \simeq [\stY_T/\agG]
\]
and that $\stY\to [\stY/\agG]$ is a $\agG$-torsor.

A proper definition of the category fibered in groupoids $[\stY/\agG]$ is the following. An object of $[\stY/\agG]$ over an $R$-scheme $T$ is a pair $(u,v)$ where $u\in \stX(T)$ and $v$ is a section of $[\stY_u/\agG]\to T$, where $\stY_u \to T$ is the base change of $\stY\to \stX$ along $u\colon T\to \stX$.

If the action of $\agG$ on $\stY$ is free, that is, it is free on all the $\stY_T$, then $[\stY_T/\agG]$ equals the quotient sheaf $\stY_T/\agG$. In this situation we will simply write $\stY/\agG$ in place of $[\stY/\agG]$.
This is coherent with the idea that if $\stY$ is a collection of sheaves indexed by objects of $\stX$ (more precisely a pseudo-functor) then $\stY/\agG$ is the collection of quotient sheaves.
\end{remark}

\begin{remark}
A $\agG$-torsor is a faithfully flat and affine map $\torP\to \stX$ together with an action of $\agG$ over $\stX$. Moreover the action of $\agG$ is free. If $\agH$ is a flat and affine subgroup of $\agG$ in particular we can consider the quotient $\torP/\agH \to \stX$, which, according to \autoref{rem:quotients}, coincides with $[\torP/\agH] \to \stX$.

We warn the reader that this notation, in some concrete situation, may lead to confusion, which we will try to avoid. For example $\Spec R \to \BB_R \agG$ is a $\agG$-torsor, but we would never write $\Spec R/\agG$ for $[\Spec R/\agG]=\BB_R \agG$.
% 
% over $T$ for all objects -torsor $\torP$, the action of $\agG$ on $\torP$ is always free,
% so the naive quotient $\torP/\agG$ and the stack quotient $[\torP/\agG]$ coincide.
% Here we will mainly write $\torP/\agG$ to remind of this fact, and later shift tacitly to the other notation, when quotients by not necessarily free actions enter.
\end{remark}

\begin{remark}
\label{rem:torsor-push-pull}
Let $\alpha \colon \agH \to \agG$ be a map of flat and affine group schemes over $R$.
Then there is an induced functor $\BB\alpha \colon \BB \agH \to \BB \agG$ such that, for any $\agH$-torsor $\torQ$, there is a $\agH$-equivariant map $\torQ\to \BB\alpha(\torQ)$ universal among maps from $\torQ$ to a $\agG$-torsor.

Moreover, recall that any (equivariant) homomorphism between torsors is automatically an isomorphism.
In particular, if $\torQ$ is an $\agH$-torsor and $\torP$ a $\agG$-torsor over a common base $\stX$, and $\torQ \to \torP$ an $\agH$-equivariant map over $\stX$, then the induced map $\BB\alpha(\torQ) \to \torP$ is an isomorphism.
\end{remark}

\begin{remark}
\label{rem:torsor-quotient}
Let $\alpha\colon \agG\to \mathsf K$ be a morphism of affine group schemes over $R$ and assume that $\agG$ is flat over $R$. Set also $\agH=\ker(\alpha)$, which is a closed subscheme of $\agG$, hence affine.

Then $\alpha$ is faithfully flat if and only if $\agH$ is flat over $R$ and $\alpha$ is an fpqc epimorphism. Indeed if $\alpha$ is an fpqc epimorphism then  the monomorphism $\agG/\agH \to \mathsf K$ is an epimorphism and hence an isomorphism. Moreover $\agG\to \agG/\agH$ is an $\agH$-torsor, hence affine and faithfully flat if $\agH$ is flat.

In this situation $\agG/\agH \simeq \mathsf K$ we have  
\[
\BB\alpha(\torQ)=\torQ/\agH
\]
for a $\agG$-torsor $\torQ$ (over some base). Indeed $\agG/\agH$ acts on $\torQ/\agH$ and this space is locally trivial, turning it into a torsor. As there is a $\agG$-equivariant map $\torQ \to \torQ/\agH$ we can conclude that $\BB\alpha(\torQ)=\torQ/\agH$ by \autoref{rem:torsor-push-pull}.
\end{remark}

The content of the following lemma and corollary is already known, see for example \cite[Lem. 1.17]{Nori}.

\begin{lemma}
\label{lem:inducedtorsor}
Let $\agG$ be a flat affine group scheme over a ring $R$, $\stX$ a pseudo-algebraic fibered category over $R$, $\torP$ a $\agG$-torsor over $\stX$ and $\alpha\colon \agH \to \agG$ a subgroup.
Then $\agH$-torsors $\torQ$ together with an $\agH$-equivariant map  $\torQ \to \torP$ over $\stX$ (equivalently an isomorphism $\BB\alpha(\torQ)\simeq \torP$) correspond to sections of $\torP/\agH \to \stX$. 
Given such a section the corresponding map  $\torQ\to \torP$ fits into the following fibre square
\begin{equation}
\label{eq:reduction to structure group}
\begin{tikzcd}
\torQ \ar[r] \ar[d, two heads, "/\agH"] & \torP \ar[d, two heads, "/H"] \\ 
\stX \ar[r] & \torP/\agH
\end{tikzcd}
\end{equation}
In particular $\torP$ is induced by an $\agH$-torsor if and only if $\torP/\agH\to \stX$ has a section.
\end{lemma}

\begin{proof}
If $\torQ$ is an $\agH$-torsor and $\torQ \to \torP$ is $\agH$-equivariant, then
$\stX = \torQ/\agH \to \torP/\agH$ is a section and
\[
\torQ\to \stX\times_{\torP/\agH}\torP
\]
is a morphism of $\agH$-torsors and therefore an isomorphism.
Conversely, the $\torQ$ defined in the statement is an $\agH$-torsor over $\stX$ and the map $\torQ\to \torP$ is $\agH$-equivariant.
\end{proof}

\begin{corollary}
\label{prop:induced-torsor}
Let $1 \to \agH \xto{\alpha} \agG \xto{\beta} \agG/\agH \to 1$ be an exact sequence of flat affine group schemes over $R$ (in particular $\agH$ is a normal subgroup of $\agG$). Then there is a Cartesian diagram
\[
\begin{tikzcd}
\BB \agH \ar[r] \ar[d, "\BB\alpha"'] & \Spec R \ar[d] \\
\BB \agG \ar[r] \ar[r, "\BB\beta"] & \BB(\agG/\agH)
\end{tikzcd}
\]
%In other words,
% $\agH$-torsors correspond to $\agG$-torsors together with a trivialization of the induced $\agG/\agH$-torsor.
In other words, an $\agH$-torsor corresponds to a $\agG$-torsor together with a trivialization of the induced $\agG/\agH$-torsor.
\end{corollary}

\begin{proof}
This is a direct consequence of \autoref{lem:inducedtorsor} because $\BB\beta(\torP)\simeq [\torP/\agH]$ (see \autoref{rem:torsor-quotient}).
\end{proof}

\begin{definition}
A \emph{$\agG$-equivariant quasi-coherent sheaf} on $\stX$ is a quasi-coherent sheaf $\shF$ on $\stX$ together with an action of $\agG$. By this we mean that for all objects $x\colon T\to \stX$ the group $\agG$ acts on $x^*\shF$ and all those actions are compatible. We denote by $\QCoh^G(\stX)$ the category of $\agG$-equivariant quasi-coherent sheaves.
\end{definition}

\begin{remark}
\label{rem:equivalences-qcoh}
Let $\agG$ be a flat affine group scheme and $\stX$ be a pseudo-algebraic fibered category over $R$.
Then
\[
\QCoh(\BB_\stX \agG)  \simeq \QCoh^{\agG}(\stX).
\]
%  where the category on the right is the category of $\agG$-equivariant quasi-coherent sheaves on $\stX$, i.e.\ quasi-coherent sheaves together with an action of $\agG$.
This follows from the same result for schemes.
Taking ring objects, this equivalence extends also to an equivalence 
\[
\Alg(\BB_\stX \agG) \simeq \Alg^\agG(\stX)
\]
from the category of quasi-coherent sheaves of algebras on $\BB_\stX \agG$ to the category of $\agG$-equivariant sheaves of algebras on $\stX$.

By applying the relative spectrum (which is a contravariant equivalence) we also obtain an equivalence
%Given a pseudo-algebraic fibered category $\stX$ over $R$,
%We denote by $\Aff(\stX)$ the category of affine maps $\stY \to \stX$.
%For an affine group scheme $\agG$ over $R$,
%we write $\Aff^\agG(\stX)$ for the category affine maps $\stY \to \stX$ with an action of $\agG$ on $\stY$ over $\stX$.
\[
\Aff(\BB_\stX \agG) \to \Aff^\agG(\stX) 
\]
from the category of affine maps over $\BB_\stX\agG$ to the category of affine maps over $\stX$ together with an action of $\agG$ over $\stX$. This equivalence maps an affine map $\stZ\to \BB_\stX\agG$ to the base change $\stZ\times_{\BB_\stX\agG}\stX\to \stX$. 
A quasi-inverse is obtained by mapping an affine map $\stY\to \stX$ with an action of $\agG$ to $[\stY/\agG]\to \BB_\stX\agG$ (see \autoref{rem:quotients}). 
% The category fibered in groupoids $[\stY/\agG]$ can be defined in the following way. An object of $[\stY/\agG]$ over an $R$-scheme $X$ is a pair $(u,v)$ where $u\in \stX(X)$ and $v$ is a section of $[Y/\agG]\to X$, where $Y\to X$ is the base change of $\stY\to \stX$ along $u\colon X\to \stX$.
\end{remark}

For $\alpha \colon \agG \to \agH$ of flat affine group schemes over $R$, 
denote by $\BB \alpha$ also the map
\[
(\BB \alpha)_\stX \colon \BB_\stX \agG = \BB \agG \times \stX \xto{\BB \alpha \times \id_{\stX}} \BB \agH \times \stX = \BB_\stX \agH.
\]
Moreover, we will use from \autoref{rem:equivalences-qcoh} the composition
\[
\hat \alpha_* \colon \Alg^\agH(\stX) \isom \Alg(\BB_\stX\agH) \xto{\BB \alpha_*}  \Alg(\BB_\stX \agG) \isom \Alg^\agG(\stX).
\]

\begin{proposition}
\label{prop:superwichtig}
Let $1 \to \agH \xto{\alpha} \agG \xto{\beta} \agG/\agH \to 1$ be a short exact sequence of flat affine group schemes over $R$.
Let $\stX$ be a pseudo-algebraic fibered category over $R$.
Then for $\gsR \in \Alg^\agH(\stX)$ and $\gsS \coloneqq \hat\alpha_* \gsR$, there is an isomorphism
\[
[\Spec_\stX \gsR/\agH] \simeq [\Spec_\stX \gsS /\agG]
\]
over $\BB_\stX \agG$.
\end{proposition}

\begin{proof}
The map $\BB\alpha \colon \BB_\stX \agH \to \BB_\stX \agG$ is affine, 
because its base change along (the relative) fpqc covering $\stX \to \BB_\stX \agG$ is $(\agG/\agH)\times \stX \to \stX$ (the map $\agG/\agH \to \BB \agH$ is given by the $\agH$-torsor $\agG \to \agG/\agH$).

Let $\overline \gsR$ and $\overline \gsS$ be the sheaves of algebras over $\BB_\stX \agH$ and $\BB_\stX \agG$, respectively, corresponding to $\gsR$ and $\gsS$.
By construction, the maps
\[
p\colon [\Spec_\stX \gsR/\agH] \to \BB_\stX \agH\text{ and } q\colon [\Spec_\stX \gsS /\agG] \to \BB_\stX \agG
\]
are affine and 
\[
\overline \gsR = p_* \OO_{[\Spec_\stX \gsR/\agH]}\text{ and } \overline \gsS = q_* \OO_{[\Spec_\stX \gsS /\agG]}.
\]
Moreover the affine map
\[
[\Spec_\stX \gsR/\agH] \xto p \BB_\stX \agH \xto{\BB\alpha} \BB_\stX \agG
\]
corresponds to
\[
(\BB\alpha \circ p)_* \OO_{[\Spec_\stX \gsR/\agH]}
= q_* \OO_{[\Spec_\stX \gsS/\agG]}.
\]
Thus $\BB\alpha \circ p$ and $q$ are isomorphic as affine maps over $\BB_\stX\agG$ as required.
\end{proof}

\subsection{Families of line bundles as torsors by diagonalisable group schemes}

In this section, we translate between families of line bundles and torsors.
This is a very general result which apply to any pseudo-algebraic fibered category $\stX$. For this reason we are going to consider diagonalizable groups as group scheme over $\Spec \ZZ$ and apply the results of \autoref{sec:torsors} with $R=\ZZ$: the category $\stX$ is defined over $\ZZ$ in the sense that $\stX\to \Sch=\Sch/\ZZ$ is the structure morphism.  

\begin{definition}
Let $G$ be an abelian group and let $\ZZ[G]$ be its group algebra. We denote by
\[
\D{G} \coloneqq \Spec \ZZ[G]
\]
the \emph{Cartier dual group scheme} of $G$ over $\ZZ$.
\end{definition}

\begin{remark}
Observe that $\D{G}$ co-represents the functor
\[
\Sch \to \AbGrps,\ X \mapsto \Hom_{\Grps}(G,H^0(\OO_X^*)).
\]

Indeed for any ring $A$ one has
\[
\Hom_{\Grps}(G, A^*) \simeq\Hom_{\ZZ\mhyph\Algs}(\ZZ[G], A).
\]
\end{remark}

\begin{remark}
The association $G\longmapsto \D{G}$ defines a contravariant, fully faithful and exact functor from the category of abelian groups to the category of affine group schemes over $\ZZ$. 
\end{remark}

\begin{example}
The multiplicative group scheme is
$\Gm \coloneqq \D{\ZZ} = \Spec \ZZ[t,t^{-1}]$.
\end{example}

\begin{remark}
\label{rem:equivalences-diag}
Let $G$ be an abelian group and $\stX$ be a pseudo-algebraic fibered category. Then $\QCoh(\BB_\stX \D{G})$ is equivalent to the category of $G$-graded quasi-coherent sheaves: if $\D{G}(R)\ni \chi\colon G\to R^*$ and $t\colon \Spec R\to \stX$ then
\[
\gsR=\bigoplus_g \gsR_g\text{ and } \chi \cdot x = \chi(g) x \text{ for } x \in t^*\gsR_g .
\]
Hence we can extend the equivalences of \autoref{rem:equivalences-qcoh} to include also the category of $G$-graded quasi-coherent sheaves of algebras on $\stX$, which we denote by $G\mhyphen\GrAlg(\stX)$:

\[
\begin{tikzcd}
G\mhyphen\GrAlg(\stX) \ar[r, equal, "\sim"] \ar[d, equal, "\wr"'] & \Alg(\BB_\stX \D{G}) \ar[d, equal, "\wr"] \\
\Aff^{\D{G}}(\stX) \ar[r, equal, "\sim"'] & \Aff(\BB_\stX \D{G})
\end{tikzcd}
\]
where the vertical equivalences are contravariant.
\end{remark}

\begin{proposition}
\label{prop:torsor-diag}
Let $G$ be an abelian group, $\stX$ be a pseudo-algebraic fibered category and $\pi\colon \torP\to \stX$ be an affine map with an action of $\D{G}$ corresponding to the $G$-graded sheaf of algebras
\[
\pi_*\OO_\torP =\gsR=\bigoplus_g \gsR_g.
\]
Then $\torP\to \stX$ is a $\D{G}$-torsor if and only if all maps
\[
\OO_\stX\to \gsR_0\text{ and } \gsR_g \otimes \gsR_h \to \gsR_{g+h}
\]
are isomorphisms. In this case all sheaves $\gsR_g$ are invertible on $\stX$. 
\end{proposition}

\begin{proof}
This is essentially \cite[Expos\'e I, \S 4.7.3 \& Expos\'e VIII, \S 4.1]{SGA3-1}. We give the argument for the convenience of the reader.

Recall that an $\D{G}$-action on $\torP$ corresponds to a map $\torP \times \D{G} \to \torP \times_\stX \torP$ or equivalently to a map $\eta\colon \gsR\otimes \gsR\to \gsR\otimes \ZZ[G]$.
Note that $\eta$ splits into the graded pieces
\[
\gsR_g \otimes \gsR_{g'}\to \gsR_{g+g'} \otimes \ZZ e_{g'}.
\]
In particular $\torP\times \D G \to \torP\times_\stX \torP$ is an isomorphism if and only if all maps $\eta_{g,h}\colon \gsR_g \otimes \gsR_h \to \gsR_{g+h}$ are isomorphisms.

We can therefore assume that the above two equivalent conditions hold for our algebra $\gsR$.

``$\Longleftarrow$'' In order to show that $\torP\to \stX$ is a $\D G$-torsor it is enough to show that $\pi_* \OO_\torP = \gsR$ is faithfully flat.

As $\gsR_g \otimes \gsR_{-g} \simeq \gsR_0 \simeq \OO_\stX$, one can show by an elementary argument that all sheaves $\gsR_g$ and therefore $\gsR$ are flat $\OO_\stX$-modules. As $\gsR_0 \simeq \OO_\stX$, the ring $\gsR$ is faithfully flat.

``$\Longrightarrow$'' Conversely, if $\torP\to \stX$ is a $\D G$-torsor, the sheaf $\gsR$ is fpqc locally on $\stX$ isomorphic to $\ZZ[G]$ as $G$-graded algebra. 
In particular all graded pieces $\gsR_g$ are fpqc locally invertible. By descent all sheaves $\gsR_g$ are therefore invertible.

In particular $\Spec \gsR_0 \to \stX$ is a cover of degree $1$, which is known to be an isomorphism, that is, $\OO_\stX\to \gsR_0$ is an isomorphism.
\end{proof}

The next step in this section is a translation of \autoref{prop:induced-torsor} into the language of families of line bundles.

\begin{proposition}
\label{rem:torsor-families}
Let $G$ be an abelian group. 
The assignment 
$\famF\mapsto \RS_\famF$
mapping a $G$-family to its Cox ring %(as in \autoref{def:cox sheaf})
defines an equivalence between the category of $G$-families of line bundles on $\stX$ and the category $\BB \D G(\stX)$ of $\D G$-torsors over $\stX$. 

If $\alpha\colon H\to G$ is a map of abelian groups then the functor 
$$
\BB \D \alpha\colon \BB \D G(\stX) \to \BB \D H(\stX)
$$ 
maps a $G$-family $\famF=(\famF_g)_{g\in G}$ to the $H$-family $(\famF_{\alpha(h)})_{h\in H}$.
\end{proposition}
\begin{proof}
The first part is just \autoref{prop:torsor-diag}. For the second, given a $G$-family $\famF$ there is an $H$-graded map
\[
\bigoplus_{h\in H} \famF_{\alpha(h)} \to \bigoplus_{g\in G} \famF_g
\]
By \autoref{rem:torsor-push-pull} it follows that $\Spec(\bigoplus_{h\in H} \famF_{\alpha(h)})$ is the induced $\D H$-torsor, as required.
\end{proof}

\begin{example}
\label{exa:free-families}
For $G = \ZZ$, we have $\BB \D{G} = \BB \Gm$ and elements of $\BB \Gm(\stX)$ correspond to $\ZZ$-families of line bundles, which are completely determined by a line bundle on $\stX$. So $\BB \Gm(\stX)$ is equivalent to the category of line bundles on $\stX$ with isomorphisms as arrows.

More generally, if $G = \ZZ^{(I)}$ for a set $I$ it follows that $\D{G}\simeq \Gm^I$ and $\BB \D{G} \simeq (\BB \Gm)^I$. So, by \autoref{prop:torsor-diag}, the category of $G$-families of line bundles on $\stX$ is equivalent to the category of collections of line bundles indexed by $I$, where arrows between them are collections of isomorphisms. The correspondence takes a $G$-family to its evaluation on a basis of $G$.
\end{example}

\begin{remark}
\label{rem:push-pull-diag}
If $\alpha\colon H\to G$ is a map of abelian groups we get maps $\D{\alpha} \colon \D{G}\to \D{H}$ and $\BB \D{\alpha} \colon \BB_\stX \D{G}\to \BB_\stX \D{H}$. 

One can check directly that $(\BB\D \alpha)^*$ and $(\BB\D \alpha)_*$ acts on $\D G$-equivariant quasi-coherent sheaves interpreted as $G$-graded sheaves as
\[
\shF = \bigoplus_{h \in H} \shF_h \longmapsto \BB \D{\alpha}^*(\shF) = \bigoplus_{g \in G} \left(\bigoplus_{h\in \alpha^{-1}(g)} \shF_h\right)
\]
and
\[
\shG=\bigoplus_{g\in G} \shG_g \longmapsto \BB \D{\alpha}_*(\shG) = \bigoplus_{h\in H} \shG_{\alpha(h)},
\]
respectively.

If $\shF$ or $\shG$ are sheaves of algebras then $\BB \D{\alpha}^*(\shF)$ or $\BB \D{\alpha}_*(\shG)$ are sheaves of algebras in the obvious way, respectively. 

If $\alpha \colon H\into G$ is injective with cokernel $G/H$, so that $\D{G/H}$ is the kernel of the quotient map $\D \alpha\colon \D{G}\to \D{H}$, then
\[
\BB \D{\alpha}_*(\shG) = \shG_{|H} \coloneqq \bigoplus_{h\in H} \shG_{\alpha(h)} = \shG^{\D{G/H}}.
\]
\end{remark}

\begin{proposition}
\label{lem:induced-family}
Let $0 \to H \xto{\alpha} G \xto{\beta} G/H \to 0$ be an exact sequence of abelian groups and $\stX$ a pseudo-algebraic fiber category.
A $G/H$-family of line bundles on $\stX$ corresponds naturally to a $G$-family $\famF$  of line bundles on $\stX$ together with a trivialization of $\famF_{|H}$.
\end{proposition}

\begin{proof}
This is just a translation of \autoref{prop:induced-torsor} using \autoref{rem:torsor-families}.
\end{proof}

\begin{remark}
\label{rem:induced-explicit}
In the situation of \autoref{lem:induced-family}, we make the induced $G/H$-family explicit, using \autoref{rem:push-pull-diag}.
A trivialisation of the $H$-family $\famF_{|H}$ is a  collection of isomorphisms $\xi_h \colon \famF_h \to \OO_\stX$ for $h\in H$ such that the morphism
\[
\xi\colon (\RS_\famF)^{\D{G/H}}=\RS_{\famF_{|H}}=\bigoplus_{h \in H} \famF_h\to \OO_\stX
\]
is an $\OO_\stX$-algebra morphism. 
The $\D{G/H}$-torsor induced by $\famF$ and the trivialization $\xi$ is given by the fibre product in \eqref{eq:reduction to structure group} of \autoref{lem:inducedtorsor}.
The corresponding $G/H$-graded sheaf of algebras $\RS$ is therefore given by
\[
\RS = \RS_\famF\otimes_{\RS_{\famF_{|H}}}\OO_\stX.
\]
More explicitly, by defining the ideal
\[
\gsI \coloneqq \ker(\xi)=\langle \xi_h^{-1}(1)-1 \mid h \in H \rangle < \RS_{\famF_{|H}},
\]
we obtain that
\[
\RS = \RS_\famF\otimes_{\RS_{\famF_{|H}}}\OO_\stX = \RS_\famF/\gsI \RS_\famF.
\]
Finally, for all $u \in G/H$,
\[
(\RS_\famF/\gsI_\xi \RS_F)_u = \frac{\bigoplus_{g \in \beta^{-1}(u)} \famF_g}{\gsI_\xi \bigoplus_{g \in \beta^{-1}(u)} \famF_g} ,
\]
where  $\beta\colon G\to G/H$ is the projection.
\end{remark}

\section{Existence and uniqueness of Cox rings}
\label{sec:coxrings}

In this section, we first show existence and uniqueness of Cox rings of invertible sheaves. 
From this we obtain analogous results for the Cox rings of reflexive sheaves.
Finally, we compare these results to the construction in \cite{Coxrings}.

\subsection{The Cox ring of invertible sheaves}
\label{sec:piccox}

\begin{proposition}
\label{prop:free-res}
Let $G$ be an abelian group.
Then there is a short exact sequence 
$0 \to K_1 \to K_0 \to G \to 0$ with $K_1$ and $K_0$ free abelian groups.
\end{proposition}

\begin{proof}
Consider $K_0 = \ZZ^{(G)}$ and the natural surjection $\ZZ^{(G)} \to G$.
Its kernel $K_1$ is free, since any subgroup of a free abelian group is free, see, for example, \cite[Thm. I.7.3]{Lang}.
\end{proof}

If $\stX$ is a pseudo-algebraic fibred category, then we denote by $\overline{\Pic(\stX)}$ the full subcategory of $\QCoh(\stX)$ of invertible sheaves
and with $\Pic(\stX)$ the set of their isomorphism classes.

For $\famF$ a $G$-family of line bundles on $\stX$, we denote by 
\[ \alpha_\famF\colon G \to \Pic(\stX), \quad  g \mapsto [\famF_g]\] the morphism induced by $\famF$, which is a group homomorphism.
We say that a $G$-family $\famF$ \emph{lifts} a group homomorphism $\mu \colon G \to \Pic(\stX)$ if $\famF_g \simeq \mu(g)$ for all $g \in G$.

\begin{lemma}
\label{lem:lifting-free}
Let $K$ be a free abelian group and $\alpha\colon K\to \Pic(\stX)$ be a group homomorphism. Then the morphism $\alpha$ lifts to a unique $K$-family up to isomorphism.
\end{lemma}
\begin{proof}
This follows immediately from \autoref{exa:free-families}.
\end{proof}

\begin{theorem} 
\label{thm:families-extension}
Let $H<G$ be abelian groups and $\famF$ be an $H$-family of line bundles on $\stX$.
For a group homomorphism $\mu\colon G \to \Pic(\stX)$ which extends $\alpha_\famF\colon H \to \Pic(\stX)$ holds:
\begin{itemize}
\item there is a $G$-family of line bundles $\famG$ lifting $\mu$ and extending $\famF$, that is with $\famG_{|H} \simeq \famF$;
\item the set of isomorphism classes of $G$-families extending $\famF$ and lifting $\mu$ is isomorphic to
\[
\Ext^1(G/H,H^0(\OO^*_\stX)).
\]
\end{itemize}
\end{theorem}

\begin{proof}
We first show existence.
Let $0 \to K_1 \to K_0 \to H \to 0$ and $0 \to K_1' \to K_0' \to G/H \to 0$ be free resolutions of $H$ and $G/H$, respectively (see \autoref{prop:free-res}).
By the horseshoe lemma, these two resolutions fit together to form a commutative diagram, where also the columns are short exact sequences:
\[
\begin{tikzcd}
0 \ar[r] & K_1 \ar[r] \ar[d] & K_0 \ar[r] \ar[d] & H \ar[r] \ar[d] & 0 \\
0 \ar[r] & K_1 \oplus K_1' \ar[r] \ar[d] & K_0 \oplus K_0' \ar[r] \ar[d] & G \ar[r] \ar[d] & 0 \\
0 \ar[r] & K_1' \ar[r] & K_0' \ar[r] & G/H \ar[r] & 0 \\
\end{tikzcd}
\]
%In what follows we use \autoref{lem:lifting-free} several times without mentioning.

Using the surjection $K_0 \to H$, we can turn the $H$-family $\famF$ into a $K_0$-family $\tilde \famF$. 
Now extend  $\tilde \famF$ to  a $K_0 \oplus K_0'$-family $\tilde \famG$, lifting the homomorphism $K_0 \oplus K'_0 \to G \to \Pic(\stX)$ %: this can be done by fixing line bundles for the generators of $K'_0$ 
as in \autoref{lem:lifting-free}.
Both $K_0$-families $\tilde \famG_{|K_0}$ and $\tilde \famF$ lift the same homomorphism $K_0\to \Pic(\stX)$. Thus we can choose an isomorphism $\tilde \famG_{|K_0}\simeq\tilde \famF$, again by \autoref{lem:lifting-free}. 
Using this isomorphism we also get that $\tilde \famG_{|K_1}\simeq \tilde \famF_{|K_1}$ are isomorphic to the trivial $K_1$-family.
Since $\tilde \famG_{|K_1'}$ lifts the trivial morphism $K_1'\to \Pic(\stX)$, it is isomorphic to the trivial $K'_1$-family as well. The two trivializations just defined give an isomorphism of $\tilde \famG_{|(K_1\oplus K_1')}$ and the trivial $(K_1\oplus K_1')$-family.
Applying \autoref{lem:induced-family} to $\tilde \famG$, there is a $G$-family $\famG$.
By construction, $\famG_{|H} \simeq \famF$.

We now prove the second part of the statement. Fix one $G$-family $\famG$ extending  $\famF$ and lifting $\mu$.
Given another lifting $\famG'$,  the $G$-family ${\mathcal N} \coloneqq {\famG'}  \otimes {\famG}^{-1}$ extends the trivial $H$-family and lifts the trivial morphism ${\mathcal N}\colon G \to \Pic(\stX)$. 
Thus we may assume that $\mu \colon G \to \Pic(\stX)$ is trivial, so $\famG_g \simeq \OO_\stX$ for all $g \in G$.
Moreover by \autoref{lem:induced-family}, we may even assume that $H=0$.

Consider now the exact sequences of groups
\begin{equation}
\label{main exact sequence}
0 \to K \to \ZZ^{(G-0)} \to G \to 0.
\end{equation}
It is easy to show that $K$ is generated by the elements $e_{g,h}=e_g+e_h-e_{g+h}$, where we set $e_0 \coloneqq 0$. 
%A $G$-family of line bundles on $\stX$ lifting the trivial morphism $G\to \Pic(\stX)$ is isomorphic to a $G$-family $\famF$ with $\famF_g=\OO_\stX$. 
Note that by our assumption on the $G$-family $\famG$, the structure morphisms $\xi_{g,h}$ are elements of $H^0(\OO_\stX)^*$ and satisfy the following relations: $\xi_{g,0}=1$, $\xi_{g,h}=\xi_{h,g}$ and $\xi_{g,h}\xi_{g+h,t}=\xi_{h,t}\xi_{g,h+t}$.

A direct computation on the relations between the $e_{g,h}$ shows that $\xi\colon K\to H^0(\OO_\stX)^*,\ e_{g,h} \mapsto \xi_{g,h}$ is a group homomorphism. Let $\famG'$ be another $G$-family such that $\famG'_g = \OO_\stX$ and $\eta \in \Hom(K,H^0(\OO_\stX)^*)$ its structure morphisms.
An isomorphism between $\famG$ and $\famG'$ is a collection of $\mu_g\in H^0(\OO_\stX)^*$ for all $0\neq g\in G$ such that
\[
\xi(e_{g,h})\mu_g\mu_h\mu^{-1}_{g+h}=\eta(e_{g,h}),
\]
which follows from the definition of a morphism and the identifications we made.
In other words, an isomorphism is given by $\mu\in \Hom(\ZZ^{(G-0)},H^0(\OO_\stX)^*)$ such that
\[
\xi(e_{g,h})\mu(e_{g,h})=\eta(e_{g,h}).
\]
Since the $e_{g,h}$ generates $K$, the above relation just means $\xi\mu_{|K}=\eta$ in $\Hom(K,H^0(\OO_\stX)^*)$. 

In conclusion the set of isomorphism classes of $G$-families lifting the trivial morphism $G\to \Pic(\stX)$ is in one-to-one correspondence with the quotient of $\Hom(K,H^0(\OO_\stX)^*)$ by the subgroup of elements which extends to $\ZZ^{(G-0)}$.
But applying $\Hom(-,H^0(\OO_\stX)^*)$ to the sequence (\ref{main exact sequence}) we get an exact sequence
\[
\Hom(\ZZ^{(G-0)},H^0(\OO_\stX)^*)\to \Hom(K,H^0(\OO_\stX)^*)\to \Ext^1(G,H^0(\OO_\stX)^*)\to 0
\]
which ends the proof.
\end{proof}

\begin{remark} \label{rem:family-uniqueness}
By the above proposition 
all $G$-families extending a given $H$-family $\famF$ and lifting a given group homomorphism $\mu$ are isomorphic if
\begin{itemize}
\item $G/H$ is free; or
\item the abelian group
$H^0(\OO_\stX^*)$ is divisible (which is equivalent to injective), for instance the group of units of an algebraically closed field,
\end{itemize}
as in both cases the $\Ext^1$-group in question is zero.
\end{remark}

\begin{corollary}
\label{prop:family-exists}
Let $\mu \colon G \to \Pic(\stX)$ be a group homomorphism.
Then the set of isomorphism classes of $G$-families lifting $\mu$ is in bijection with $\Ext^1(G,H^0(\OO_\stX^*))$. In particular there is a $G$-family of line bundles lifting $\mu$.
\end{corollary}

%\begin{proof}
%Let $0 \to K_1 \to K_0 \to G \to 0$ be short exact sequence with $K_i$ free abelian groups.
%Let $\mathcal N$ be a $K_0$-family lifting $K_0 \to G \to \Pic(\stX)$ (construction by choosing elements for a basis).
%Dualising the above sequence we get
%\[
%0 \to \D{G} \to \D{K_0} \to \D{K_1} \to 0
%\]
%Let $\torP$ be the $\D{K_0}$-torsor corresponding to $\mathcal N$
%\[
%\torP = \Spec_\stX \bigoplus_{k \in K_0} N_k.
%\]
%Consider the $\D{K_1}$-torsor $\stackquot{\torP}{\D{G}} = \Spec_\stX(\bigoplus_{k \in K_1} N_k)$. The corresponding $K_1$-family $\mathcal N_{|K_1}$ satisfies the assumption of \autoref{lem:trivialfamily}, so there are isomorphisms $\chi(k) \in \Hom(\OO_\stX,\mathcal N_k)$ for all $k \in K_1$.
%Then the map $\stX \to \stackquot{\torP}{\D{G}}$ induced from 
%\[
%\sigma \colon \bigoplus_{k\in K_1} \mathcal N_k \xto{\chi^{-1}} \bigoplus_{k \in K_1} \OO_\stX \xto{\Delta^*} \OO_\stX
%\]
%is a section.
%By \autoref{lem:inducedtorsor} there is a $\D{G}$-torsor $\torQ$ inducing $\torP$. This $\torQ$ corresponds to the desired $G$-family of line bundles on $\stX$.
%\end{proof}

\begin{definition} 
Let $\stX$ be a pseudo-algebraic fibered category.
A \emph{$\Pic$-Cox ring} $\PicRR(\stX)$ of $\stX$ is the Cox ring associated with a $\Pic(\stX)$-family lifting the identity $\Pic(\stX)\to \Pic(\stX)$. 
Analogously, we define a \emph{$\Pic$-Cox sheaf} $\PicRS(\stX)$ of $\stX$.
\end{definition}

Putting together \autoref{thm:families-extension} and \autoref{rem:family-uniqueness},
we obtain the following theorem.

\begin{theorem}
\label{thm:mainthm-pic}
Let $\stX$ be a pseudo-algebraic fibered category. 
Then $\stX$ admits a $\Pic$-Cox ring $\PicRR(\stX)$.
Moreover, a $\Pic$-Cox ring is unique up to isomorphism if
\[
\Ext^1(\Pic(\stX),H^0(\OO_\stX^*)) = 0,
\]
which, in particular, holds in the cases that
\begin{itemize} 
\item $\Pic(\stX)$ is free; or 
\item $H^0(\OO_\stX^*)$ is divisible, e.g. equal to $\kk^*$ for an algebraically closed field $\kk$.
\end{itemize}
\end{theorem}

\begin{corollary}
\label{thm:piccox-welldefined}
Let $\stX$ be a pseudo-algebraic fibered category over an algebraically closed field $\kk$.
If $H^0(\OO_\stX)^* = \kk^*$ (e.g. if $\stX$ is proper, reduced and connected), then the 
$\Pic$-Cox ring $\PicRR(\stX)$
is well-defined up to graded isomorphisms, that is there exists a $\Pic(\stX)$-family lifing $\Pic(\stX)\xto \id \Pic(\stX)$ and it is unique up to isomorphism.
\end{corollary}

\begin{example}
Let us show an example of a variety $X$ over a field $\kk$ with non isomorphic $\Pic$-Cox rings. We will pick an affine variety, in which case non ($\Pic(X)$-graded) isomorphic  $\Pic$-Cox rings just means non isomorphic $\Pic(X)$-family lifting $\Pic(X)\xto \id \Pic(X)$. By \autoref{thm:families-extension}, their isomorphism classes are parametrized by
\[
\Ext^1(\Pic(X),H^0(\OO_X)^*).
\]
Moreover we will pick an affine variety with $\Pic(X)\simeq \ZZ/d\ZZ$ with $d\geq 2$. In this case we must show that 
\[
\Ext^1(\Pic(X),H^0(\OO_X)^*) \simeq H^0(\OO_X)^*/(H^0(\OO_X)^*)^d \neq 0
\]
Consider an irreducible homogeneous polynomial $F\in \kk[x_0,\dots,x_n]$ of degree $d\geq 2$ and consider the complement $Y$ of the hypersurface defining by $F$ inside $\PP^n_\kk$. One has $\Pic(Y)=\ZZ/d\ZZ$.
If $(\kk^*)^d\neq \kk^*$, then $X=Y$ has the desired property.

In the case that $\kk$ is algebraically closed (and therefore $(\kk^*)^d = \kk^*$), we can choose $X=Y\times \Gm$, where $\Gm = \Spec \ZZ[t,t^{-1}]$. 
By \cite[Ex. III.12.6]{Hartshorne}, we get that $\Pic(X) \simeq \Pic(Y) \times \Pic(\Gm) \simeq \ZZ/d\ZZ$ is still cyclic of order $d$.
But $t\in H^0(\OO_X)^*$ is not a $d$-th power.
\end{example}

\begin{example}
\label{ex:nonunique:unique}
Let $\stX$ be a pseudo-algebraic fibered category.
Note that even if $\Ext^1(\Pic(\stX),H^0(\OO_\stX^*))$ is non-zero, $\stX$ might admit a unique $\Pic$-Cox ring. By \autoref{thm:families-extension}, the $\Ext$-group is in correspondence with the isomorphism classes of $\Pic$-Cox sheaves $\PicRS(\stX)$ of $\stX$. Still, when passing to the Cox ring by taking cohomology, all non-isomorphic Cox sheaves might become isomorphic Cox rings.

For example, consider $\stX = \BB \mu_m$ over $\kk = \mathbb{F}_p$ with $p$ a prime, where we find that $\Pic(\stX) \cong \ZZ/m\ZZ$ and $H^0(\OO_\stX^*) \cong \ZZ/(p-1)\ZZ$.
If $\gcd(m,p-1)>1$ then $\Ext^1(\Pic(\stX),H^0(\OO_\stX^*))$ is non-zero, hence by \autoref{thm:families-extension} there are non-isomorphic $\Pic$-Cox sheaves on $\stX$.
But $H^0(\stX,\shL) = 0$ for any non-trivial $\shL \in \Pic(\stX)$, so $\stX$ has a unique $\Pic$-Cox ring $\RR(\stX) \cong \mathbb{F}_p$.

It would be interesting to know, whether there are smooth projective varieties that give rise to similar examples.
\end{example}

We close this section with an example, that makes use of the general setting of \autoref{thm:mainthm-pic}.

\begin{example}
\label{ex:logarithm}
In \cite{TalpoVistoli} the authors attached to any logarithmic scheme $X$ an infinite root stack $\sqrt[\infty]{X}$, which is a limit of classical root stacks. 
Although $\sqrt[\infty]{X}$ is not algebraic in general, it admits an fpqc atlas by a scheme, that is, it is pseudo-algebraic in the sense of \autoref{def:pseudo-algebraic}. 
So \autoref{thm:mainthm-pic} applies to those general stacks, yielding a $\Pic$-Cox ring of $\sqrt[\infty]{X}$. 
Notice that $H^0(\OO_X) \cong H^0(\OO_{\sqrt[\infty]{X}})$, so that uniqueness of this Cox ring follows if $H^0(\OO_X)^*=\kk^*$ for an algebraically closed field $\kk$. 
\end{example}

\subsection{The Cox ring of reflexive sheaves}
\label{sec:refcox}

In this section, we extend \autoref{thm:families-extension} for families of reflexive sheaves of rank one. This will require additional conditions on $\stX$. 

If $\stX$ is an algebraic stack, we denote with $\overline{\Refl_1(\stX)}$ the subcategory of $\QCoh(\stX)$ of reflexive sheaves of rank $1$ (on the generic points of an atlas)  and with $\Refl_1(\stX)$ the set of their isomorphism classes. 

From now on let $\stX$ be a  Noetherian, normal and excellent algebraic stack and $\stU$ be its regular locus which is open and dense.
Notice that the complement of $\stU$ in $\stX$ has codimension at least $2$. We refer to \cite[\S 6]{Osserman} for the notion of codimension for algebraic stacks.

\begin{lemma}
\label{lem:restrictsmooth}
Let $j\colon \stV \into \stX$ be an open immersion whose complement has codimension at least $2$.
Then the restriction of sheaves
\[ 
j^* \colon \RRef_1(\stX) \to \RRef_1(\stV), \quad \shL \mapsto \shL_{|\stV} 
\]
is an equivalence of categories with $j_*$ its inverse.  %$\shL \mapsto (j_* \shL)\twovee$.
In particular,
\[
\RRef_1(\stX) \to \RRef_1(\stU)=\PPic(\stU), \quad \shL \mapsto \shL_{|\stU}
\]
is well defined and an equivalence. Moreover, $\Refl_1(\stX)$ is an abelian group isomorphic to $\Pic(\stU)$: given $\shL,\shL'\in \Refl_1(\stX)$ their product is $(\shL \otimes \shL')\twovee$ and the inverse of $\shL$ is $\shL^\vee$.
\end{lemma}

\begin{proof}
Let us show that $j^*$ is essentially surjective. Let $\shL \in \PPic(\stV)$. Since $j_*\shL$ is a quasi-coherent sheaf with $(j_*\shL)_{|\stV}\simeq \shL$, 
from \cite[Cor. 15.5]{LM} there exists a coherent sheaf $\shF$ on $\stX$ 
such that $\shF_{|\stV}\simeq \shL$. Then the sheaf $\shF\twovee$ is reflexive by \cite[Cor 1.2]{HaSRS} and restricts to $\shL$.

By \cite[Prop 1.6]{HaSRS} and the hypothesis on the codimension it follows that $\shL\to j_*j^*\shL$ is an isomorphism for reflexive sheaves on $\stX$. In particular $j_*\colon \RRef_1(\stV)\to \RRef_1(\stX)$ is well defined and a quasi-inverse of $j_*\colon \RRef_1(\stX) \to \RRef_1(\stV)$.

For the second part one has $\RRef_1(\stU)=\PPic(\stU)$ because $\stU$ is regular. The last statement is an easy consequence of \cite[Cor 1.2]{HaSRS}.
\end{proof}

\begin{corollary}
\label{cor:restriction}
Let $j\colon \stU \into \stX$ be the inclusion of the regular locus.
The restriction induces an equivalence between the category of $G$-families of sheaves on $\stX$ and the category of $G$-families of invertible sheaves on $\stU$.
\end{corollary}

\begin{proof}
By \autoref{rem:pullback-families}, the restriction preserves families as $j \colon \stU \into \stX$ is flat.
Let $\famF$ be a $G$-family of invertible sheaves on $\stU$. Pushing forward $\famG=j_*\famF$ we get all the data of a family of sheaves on $\stX$. The only condition missing is the fact that $\famG_g\to \SHom_\stX(\famG_{g'},\famG_{g+g'})$ is an isomorphism. Since this map restrict to the corresponding map for the family $\famF$, it is enough to observe that
\[
\SHom_\stX(\famG_{g'},\famG_{g+g'}) \simeq j_* \SHom_\stU(\famF_{g'},\famF_{g+g'})
\]
is a reflexive sheaf of rank $1$.
\end{proof}

\begin{corollary}
\label{cor:ref-cl-same}
Let $X$ be a Noetherian, normal and excellent scheme. Then
\[
\Refl_1(X)\simeq \Cl(X) \simeq \Pic (U)
\]
where $\Cl(X)$ is the group of Weil divisors modulo equivalence and $U$ is the regular locus of $X$.
\end{corollary}
\begin{proof}
Indeed one has $\Cl(X)=\Cl(U)=\Pic(U)$ and $\Refl_1(X)=\Pic(U)$ by \autoref{cor:restriction}.
\end{proof}

\begin{definition} 
Let $\stX$ be a noetherian, normal and excellent algebraic stack.
A \emph{$\Refl_1$-Cox ring} $\RefRR(\stX)$ of $\stX$ is a Cox ring associated with a $\Refl_1(\stX)$-family lifting the identity of $\Refl_1(\stX)$.
Analogously, we define a \emph{$\Refl_1$-Cox sheaf} $\RefRS(\stX)$ of $\stX$.
\end{definition}

As an application we get the following theorem.

\begin{theorem} 
\label{prop:picref}
If $\stX$ is a noetherian, normal and excellent algebraic stack then \autoref{thm:mainthm-pic} and \autoref{thm:piccox-welldefined} hold if we replace $\Pic$ by $\Refl_1$. 
\end{theorem}

\begin{proof}
One has $\Pic(\stX)=\Pic(\stU)$ and $H^0(\stX,\OO_\stX)^*=H^0(\stU,\OO_\stU)^*$. So everything follows from \autoref{cor:restriction}.
\end{proof}

We end this subsection by a comparison between $\Refl_1$-Cox sheaf and $\Pic$-Cox sheaf.
Recall that for an $\agH$-torsor $\torP$ over a pseudo-algebraic fibered category $\stX$, we have $\stX = [\torP/\agH]$. In particular for the $\Pic$-Cox ring we get that
\[
\stX = \stackquot{\Spec_\stX \PicRS(\stX)}{\D{\Pic(\stX)}}.
\]

\begin{proposition}
\label{prop:Xcan-proper}
Let $\stX$ be a Noetherian, normal and excellent algebraic stack and assume that $T=\Refl_1(\stX)/\Pic(\stX)$ is finite. 
%Denote by $\RS(\stX)$ the $\Refl_1$-Cox sheaf. 
Then there is a factorization
\[
[\Spec_\stX \RefRS(\stX)/\D{\Refl_1(\stX)}]\to \stY \to \stX,
\]
where the first map is finite and $\stY\to \stX$ is a (relative) $\D T$-gerbe. In particular, the above map is proper.
\end{proposition}

\begin{proof}
Note that we can realize $\RS' \coloneqq \PicRS(\stX)$ as a graded subsheaf of algebras of $\RS \coloneqq \RefRS(\stX)$, since $\Pic(\stX) \subseteq \Refl_1(\stX)$.
%Let $\RS'(\stX)$ be the subsheaf of algebras of $\RS(\stX)$ obtaining by taking the graded pieces indexed by $\Pic(\stX)\subseteq \Refl_1(\stX)$. 
%By construction $\RS'(\stX)$ is a $\Pic$-Cox sheaf. 
There is an exact sequence
\[
1 \to \D T \to \D{\Refl_1(\stX)}\to \D{\Pic(\stX)}\to 1
\]
and, moreover, we have a Cartesian diagram
\[
\begin{tikzcd}
\Spec_\stX \RS \ar[r, "\alpha"] \ar[d] & \Spec_\stX \RS' \ar[d, shift right=4] \\
\stackquot{\Spec_\stX \RS}{\D{\Refl_1(\stX)}} \ar[r, "\overline \alpha"] & \stackquot{\Spec_\stX \RS'}{\D{\Refl_1(\stX)}} \eqqcolon \stY
\end{tikzcd}
\]
because $\alpha$ is $\D{\Refl_1(\stX)}$-equivariant. 
In particular, the induced map $\overline \alpha$ is finite, if $\alpha$ is finite.

To see that $\alpha$ is finite, we can actually assume that $\stX$ is affine (but keeping the original groups $\Pic(\stX)$ and $\Refl_1(\stX)$). 
If we consider $\RS$ as an $\RS'$-module, then it is generated by $\OO_\stX$-generators of the $\RS_q$, where $q$ runs through a system of generators of $T$ in $\Refl_1(\stX)$. For this note that if $q\in \Refl_1(\stX)$ and $t\in \Pic(\stX)$ then 
$\RS_{t+q} \simeq \RS'_t \otimes \RS_q$
by \autoref{rem:reflexive-torsor}. Now consider the Cartesian diagram
\[
\begin{tikzcd}
\stY = [\Spec_\stX \RS'/\D{\Refl_1(\stX)}] \ar[d, shift left=4] \ar[r] & \BB \D{\Refl_1(\stX)} \ar[d] \\
\stX = [\Spec_\stX \RS'/\D{\Pic(\stX)}] \ar[r] & \BB \D{\Pic(\stX)}
\end{tikzcd}
\]
The equation in the bottom left corner holds since $\Spec_\stX \RS'$ is a $\D{\Pic(\stX)}$-torsor over $\stX$. 
Hence $\stY\to \stX$ is a base change of $\BB \D{\Refl_1(\stX)} \to \BB\D{\Pic(\stX)}$, which is a relative $\D T$-gerbe.

For the last claim we observe that $\D T \to \Spec \ZZ$ is proper. Indeed it is separated because the diagonal of $\D T$ is (fppf) locally of the form $\D T \to \Spec \ZZ$, which is finite. Moreover $\BB \D T\to \Spec \ZZ$ is a universal homeomorphism, hence universally closed.
\end{proof}

\subsection{Divisorial approach for the Cox ring}
\label{sec:confrontation}

In this section, we show that our construction of a Cox ring is a generalisation of the contruction in \cite[\S 1.4]{Coxrings} for varieties.

Let $X$ be an integral, Noetherian, normal scheme and choose a short exact sequence
\[
0\to K_1\to K_0 \to \Cl(X) \to 0
\]
of abelian groups where $K_0$ and, therefore, $K_1$ are free, where $\Cl(X)$ is the group of Weil divisors modulo equivalence. Recall that $\Cl(X)\simeq \Refl_1(X)$ and they are both isomorphic to $\Cl(U)\simeq \Pic(U)$, where $U$ is the regular locus of $X$ (see \autoref{cor:ref-cl-same}).

We can lift $K_0\to \Cl(X)$ to a map $E\colon K_0\to \WDiv(X), k \mapsto E_k$, where $\WDiv(X)$ is the group of Weil divisors on $X$. This defines a $K_0$-family $\famG$ of reflexive sheaves of rank one with components
\[
\famG_k = \OO_X(E_k),
\]
where the multiplication is given by the usual multiplication of rational sections in $\kk(X)$.
The Cox sheaf of algebras associated with the $K_0$-family $\famG$ is
\[
\SS(X) \coloneqq \RS_\famG = \bigoplus_{k \in K_0} \famG_k
\]
Note that $\SS(X)$ is the sheaf of divisorial algebras for the surjection $K_0 \onto \Cl(X)$,
as introduced in \cite[\S 1.3.1]{Coxrings}.

By construction $\famG_{|K_1}$ is trivial.
A trivialization $\xi_k\colon \famG_k\to \OO_X$ is induced by a group homomorphism
\[
\zeta\colon K_1 \to \kk(X)^*
\]
such that $\div(\zeta(k))=E_k$, so that $\xi_k(f)=f/\zeta(k)$. 
Consider the ideal
\[
\gsI = \langle \zeta(k)-1 \mid k\in K_1\rangle \subseteq \bigoplus_{k\in K_1}\OO_X(E_k)
\]
where $\zeta(k)\in \OO_X(E_k)$, more precisely $\OO_X(E_k)=\zeta(k)\OO_X$ and
\[
\RS = \RS_\famG / \gsI \RS_\famG.
\]
This gives exactly the Cox sheaf $\RS(X)$ as constructed in \cite[\S 1.4.2]{Coxrings}.
We claim that $\RS$ is the 
$\Refl_1$-Cox sheaf $\RefRS(X)$.
%Cox sheaf attached to a $\Refl_1(X)$-family of reflexive sheaves on $X$ lifting the identity of $\Refl_1(X)$.
Note that here $\Refl_1(X) = \Cl(X)$.

We make use of \autoref{cor:restriction}. Set $\famF$ for the $K_0$-family of line bundles on the regular locus $U$ restriction of $\famG$. In particular $\RS_\famF=(\RS_\famG)_{|U}$. The trivialization of $\famG_{|K_1}$ defines a trivialization of $\famF_{|K_1}$.

By \autoref{rem:induced-explicit} we obtain a $\Pic(U)$-family $\overline \famF$ of line bundles on $U$ and therefore a $\Cl(X)$-family $\overline \famG$ of reflexive sheaves on $X$ such that $\overline \famG_{|U}=\famF$. Again by  \autoref{rem:induced-explicit}, we have
\[
(\RS_{\overline \famG})_{|U} \simeq \RS_{\overline \famF} \simeq \RS_{|U}
\]

In order to conclude that $\RS_{\overline \famG} \simeq \RS$, it is enough to notice that $\RS$ is $\Cl(X)$-graded and each graded piece is isomorphic to a graded piece of $\RS_\famG$ and therefore reflexive of rank $1$.

More precisely we have 
\begin{equation}
\label{eq:divisorial algebras}
(\RS_\famG)_k=\famG_k = \OO_X(E_k) \simeq \RS_{\pi(k)}\text{ for } k \in K_0
\end{equation}
where $\pi\colon K_0\to \Cl(X)$.

\section{Applications to Mori dream stacks}
\label{sec:application}

In this section, let $\kk$ be an algebraically closed field of characteristic zero.
Let $X$ be a normal integral scheme of finite type over $\kk$ such that
$X$ has only constant invertible functions and a finitely generated divisor class group. 

As in \autoref{sec:confrontation},  let $\RS(X) \coloneqq \RefRS(X)$ be the $\Refl_1$-Cox sheaf and $\SS(X)$ the sheaf of divisorial algebras appearing in the construction of $\RR(X) \coloneqq \RefRR(X)$, after choosing some surjection $K_0 \onto \Cl(X)$ and a lift $K_0 \to \WDiv(X)$.   

\begin{proposition}
\label{thm:rs}
In the above situation, $\SS(X) \onto \RS(X)$ induces an equivalence of quotient stacks
\[  
\stackquot{\Spec_X \RS(X)}{\D {\Cl(X)}} 
\simeq  
\stackquot{\Spec_X \SS(X)}{\D{K_0}}.
\] 
\end{proposition}

\begin{proof}
We  apply \autoref{prop:superwichtig} to the short exact sequence 
\[
1\to \D {\Cl(X)} \xto{\alpha} \D{K_0} \to \D{K_1} \to 1
\]
and to $\gsR = \RS(X)$. By \autoref{rem:push-pull-diag} and \eqref{eq:divisorial algebras} of \autoref{sec:confrontation}
we can conclude that $\hat \alpha_* \gsR = \SS(X)$, so that the statement follows.
\end{proof}

\begin{lemma}
\label{lem:Xcan-regular}
In the situation above,
the ring $\SR(X)$ is a unique factorization domain and $\Spec_X \SS(X)$ is a locally factorial scheme of finite type over $\kk$.

If $X$ has affine diagonal (e.g.\ separated), then $\Spec_X \SS(X) \to \Spec \SR(X)$ is an open immersion whose complement has codimension at least $2$. 
\end{lemma}

\begin{proof}
As $\SR(X)$ is constructed using a surjection $K_0 \onto \Cl(X)$, the ring $\SR(X)$ is a unique factorization domain by \cite[Thm. 1.3.3.3]{Coxrings}. Moreover, since $X$ is $\QQ$-factorial, $\Spec_X \SS(X)\to X$ is of finite type, see \cite[Proposition 1.3.2.3]{Coxrings}.

Now assume that $X$ has affine diagonal. By construction $V=\Spec_X \SS(X)$ and $Y=\Spec \SR(X)$ have the same global sections, because
\[
H^0(V,\OO_V)=H^0(X,\SS(X))=\SR(X).
\]
By \cite[Corollary 1.3.4.6]{Coxrings} the scheme  $V$ is quasi-affine, hence $V\into Y$ is an open immersion.

Assume by contradiction that the complement of $V$ in $Y$ has codimension lower than $2$. This means that there exists a prime ideal $P$ of $\SR(X)$ of height $1$ such that $P\notin V$. Since $\SR(X)$ is UFD, the prime $P$ is principal, say $P=(g)$, which implies that $V\subseteq \Spec \SR(X)_g$.
Since $H^0(\OO_V)=H^0(\OO_Y)=\SR(X)$, it would follow that $g$ is invertible, which is not the case.

Now we prove in general that $V = \Spec_X \SS(X)$ is a locally factorial scheme. 
When $X$ has affine diagonal, this follows from the fact that $V$ is an open subset of $\Spec \SR(X)$.
To see this, denote by $\pi\colon \Spec_X \SS(X)\to X$ the structure morphism.
If $U$ is an open affine subset of $X$ one has that $\SS(X)_{|U}$ is again a sheaf of divisorial algebras for $U$, because $\Cl(X)\to \Cl(U)$ is surjective. In particular
\[
\Spec_U (\SS(X)_{|U})=\pi^{-1}(U)
\]
is locally factorial as required.
\end{proof}

%\begin{remark}
% If $W\to Z$ is an affine and faithfully flat map and $W$ has affine diagonal then $Z$ has affine diagonal.
% 
% Indeed, first reduce to the case that $Z$ is quasi-compact, so that $W$ is quasi-compact, too. If $W'\to W$ is a Zariski covering with $W'$ affine, then $W' \to W$ is affine as well, because $W$ has affine diagonal. Replacing $W$ by $W'$ we can therefore assume that $W$ is affine.
% Consider the Cartesian diagram
%   \[
%\begin{tikzcd}
%W \times_Z W \ar[r] \ar[d] & W \times W \ar[d] \\
%Z \ar[r] & Z \times Z
%\end{tikzcd}
%  \]
%As $W\to Z$ is affine and $W$ is affine it follows that $W\times_Z W$ and $W\times W$ are affine. In particular $Z\to Z\times Z$ is affine after base change along the fpqc covering $W\times W\to Z\times Z$. By descent the diagonal has to be affine.
%
%As a consequence, we get a partial reverse for \autoref{lem:Xcan-regular}: if $\Spec_X \SS(X)$ is quasi-affine, then $X$ has affine diagonal.
%\end{remark}

With \autoref{thm:rs} we can compare Mori dream spaces and stacks in the spirit of \cite[Prop. 2.9]{HM2}.

\begin{definition}
A \emph{Mori dream space} over $\kk$ is an integral and normal scheme $X$ of finite type over $\kk$ such that:
\begin{itemize}
\item $X$ is $\QQ$-factorial with affine diagonal (e.g. $X$ separated),
\item $H^0(X, \OO^*_X) = \kk^*$,
\item $\Cl(X) = \Refl_1(X)$ is finitely generated as a $\ZZ$-module,
\item its $\Refl_1$-Cox ring $\RefRR(X)$ is finitely generated as a $\kk$-algebra.
\end{itemize}
\end{definition}

\begin{remark}
This definition is slightly more general than in \cite{HM1,HM2}, where we assumed a Mori dream space to be separated.
\end{remark}

Note that for a Mori dream space its Cox sheaf $\RefRS(X)$ is well defined up to (graded) isomorphism. In particular it does not depend on the choices made in \autoref{sec:confrontation}. 

\begin{definition}
Let $X$ be a Mori dream space. We call 
\[ \Xcan \coloneqq
\stackquot{\Spec_X \RefRS(X)}{\D{\Cl(X)}} 
\] 
the \emph{canonical \MD-stack} of $X$. 
\end{definition}

\begin{lemma}
\label{lem:piocan}
Let $X$ be a Mori dream space, $U$ its regular locus and $\pi\colon \Xcan \to X$ be the structure morphisms. Then $\Xcan$ is a normal stack of finite type over $\kk$, $\Refl_1(\Xcan)=\Pic(\Xcan)$ and the complement of $U=\pi^{-1}(U)$ in $\Xcan$ has codimension at least $2$.
\end{lemma}

\begin{proof}
From \autoref{lem:Xcan-regular} we see that $\Xcan$ is normal and of finite type.

Recall that $\Cl(X)=\Refl_1(X)$ by \autoref{cor:ref-cl-same}.
Let $0 \to K_1 \to K_0 \xto{\alpha} \Cl(X) \to 0$ be a short exact sequence with $K_0, K_1$ free abelian groups.
By \autoref{thm:rs}, $\Xcan$ is isomorphic to the quotient stack
\[
\stackquot{\Spec_X \SS(X)}{\D{K_0}},
\]
where as before $\SS(X)$ is the sheaf of divisorial algebras appering in the construction of $\RefRR(X)$.

Since $\Spec_X \SS(X)$ is a locally factorial scheme by \autoref{lem:Xcan-regular}, by \autoref{cor:ref-cl-same} all reflexive sheaves of rank $1$ on $\Spec_X \SS(X)$ are invertible. In particular if we denote by $f\colon \Spec_X \SS(X)\to \Xcan$ the structure map, which is a smooth atlas, for all $\shL\in \Refl_1(\Xcan)$ the sheaves $f^*\shL$ and, by descent, $\shL$ itself are invertible. Thus $\Refl_1(\Xcan)=\Pic(\Xcan)$. 

Set $p\colon \Spec_X \SS(X)\to X$. As $\Refl_1(X)=\Pic(U)$, the restriction of $p$ to $U$ is a $\D{K_0}$-torsor and therefore $\pi^{-1}(U)\to U$ is an isomorphism. By \cite[Proposition 1.3.2.8]{Coxrings} the inverse image of $U$ on $\Spec_X \SS(X)$ has complement of codimension at least $2$. As $\Spec_X \SS(X)\to \Xcan$ is a smooth atlas, by definition we have that $U=\pi^{-1}(U)$ has complement in $\Xcan$ of codimension at least $2$.
\end{proof}

\begin{remark}
From \autoref{lem:Xcan-regular} we see that $\Xcan$ is not just normal, but it has a smooth atlas $\Spec_X \SS(X) \to \Xcan$ with $\Spec_X \SS(X)$ locally factorial. 

In general, this does not imply that $\Spec_X \SS(X)$ is smooth, in which case $\Xcan$ would be smooth and therefore the canonical stack of $X$  in the sense of \cite[\S 4.1]{Fantechi-etal}.

For instance consider $X = \Spec A$ with
\[
A = \kk[x_1,\dots,x_n]/(x_1^2+\cdots + x_n^2) \text{ for }n\geq 5.
\]
The algebra $A$ is a UFD by the Klein-Nagata theorem, but not regular because the Jacobian criterion fails. 
In particular $\Cl(X)=\Pic(X)=0$ and $H^0(X,\OO_X)^*=A^*=\kk^*$, since $A$ is a $\NN$-graded domain and therefore an invertible element of $A$ has to be homogeneous of degree $0$.

In conclusion $X$ is a Mori dream space and $\Xcan=X$ is not regular.
\end{remark}

\begin{theorem}
\label{thm:piccan}
Let $X$ be a Mori dream space and $\Xcan$ its canonical \MD-stack.
Then the structure morphism $\pi \colon \Xcan\to X$ induces maps
\[
\begin{tikzcd}[row sep=0]
\PPic(\Xcan)=\RRef_1(\Xcan) \ar[r, shift right] & \RRef_1(X) \ar[l, shift right] \\
\phantom{\PPic(\Xcan)=} \shF \ar[r, mapsto] & \pi_* \shF \\
\phantom{\PPic(\Xcan)=} (\pi^*\shG)^{\vee\vee} & \shG \ar[l, mapsto]
\end{tikzcd}
\]
which are well defined and quasi-inverses of each other.

In particular, we get that $H^0(\Xcan,\OO_{\Xcan}^*)=H^0(X,\OO_X)=\kk^*$. 
Moreover, 
%for the $\Pic$-Cox sheaf $\RS(\Xcan)$ and the $\Pic$-Cox ring $\RR(\Xcan)$  of $\Xcan$, 
we have
\[
\pi_* \PicRS(\Xcan) \simeq \RefRS(X) \text{ and } \PicRR(\Xcan) \simeq \RefRR(X).
\]
\end{theorem}

\begin{proof}
Let $i\colon U=\pi^{-1}(U) \to \Xcan$ be the open immersion and $j=\pi i \colon U\to X$ the inclusion.

The equality $\Refl_1(\Xcan)=\Pic(\Xcan)$ and the fact that $U$ has complement of codimension at least $2$ in $\Xcan$ follow from \autoref{lem:piocan}.

From \autoref{lem:restrictsmooth} we have that 
\[
i_* \colon \RRef_1(U)\to \RRef_1(\Xcan) \text{ and } j_* \colon \RRef_1(U)\to \RRef_1(X)
\]
are equivalences, quasi-inverses of the corresponding restrictions. It follows that
\[
\pi_* \colon \RRef_1(\Xcan) \to \RRef_1(X)
\]
is an equivalence. Moreover if $\shG\in \RRef_1(X)$ then $(\pi^*\shG)^{\vee\vee}$ is a reflexive sheaf whose restriction to $U$ is $\shG_{|U}$. 

The last part of the statement follows from \autoref{cor:restriction}.
\end{proof}

% \begin{definition}
%  A \emph{Mori dream} stack, or MD-stack in short, over $\kk$ is an algebraic stack $\stX$ of finite type over $\kk$ and with affine diagonal such that
% \begin{enumerate}
%  \item $\Pic(\stX)$ is a finitely generated abelian group;
%  \item $H^0(\stX,\OO_\stX)^*=\kk^*$;
%  \item $\PicRR(\stX)$ is a graded factorial, normal and finitely generated $k$-algebra;
% \end{enumerate}
% 
% A \emph{Mori dream quotient} stack, or MD-quotient stack in short, over $\kk$ is a Mori dream stack $\stX$ over $k$ of the form 
% \[
% \stX \simeq [U/\D G]
% \]
% where $G$ is an abelian group and $U$ is a quasi-affine scheme with an action of the diagonalizable group $\D G$. Equivalently, it satisfies one of the equivalent conditions listed in \autoref{thm:quasi-affiness PicCox}.
% \end{definition}

\begin{remark}
\label{rem:generaliseHM2}
As for any pseudo-algebraic fibered category, we find that
\[
\Xcan = \stackquot{\Spec_{\Xcan} \PicRS(\Xcan)}{\D{\Pic(\Xcan)}}.
\]
So by its very definition, we get that $\Spec_{\Xcan} \PicRS(\Xcan) = \Spec_X \RefRS(X)$.
Moreover, by \cite[Prop. 1.6.3.3]{Coxrings}, $\Spec_X \RefRS(X) \into \Spec \RefRR(X)$ is an open immersion, whose complement is given by the \emph{irrelevant ideal} which is of codimension at least 2.

Additionally by \autoref{prop:Xcan-proper}, the map $\Xcan \to X$ is proper. 
Hence if $X$ is separated, so is $\Xcan$.

Therefore, by the last part of \autoref{thm:piccan}, we get for separated $X$ that $\Xcan$ is a \MD-quotient stack in the sense of \cite{HM2}, 
under the hypothesis that $H^0(\Spec \PicRR(\Xcan),\OO^*) = H^0(\Spec \RefRR(X),\OO^*) = \kk^*$.

This was shown already in \cite[Prop. 2.9]{HM2} with two additional assumptions:
namely that $H^0(\Spec \RefRR(X),\OO^*) = \kk^*$ (again) and $\Pic(\Spec \RefRR(X)) = 0$.

So \autoref{thm:piccan} allows to weaken the definition of \MD-quotient stack $\stX$ in \cite{HM2}: instead of asking for $H^0(\Spec \PicRR(\stX),\OO^*) = \kk^*$, the weaker condition $H^0(\stX,\OO^*)=\kk^*$ is sufficient.

As another consequence, the assumptions on $\Spec \RefRR(X)$ can be dropped in all statements in \cite{HM2}, which involve \cite[Prop. 2.9]{HM2}.
Most notable in the main theorem \cite[Thm. 3.2]{HM2} about lifting maps between Mori dream spaces, and \cite[Thm. 4.3]{HM2} where a classification of Mori dream quotient stacks is given in terms of root constructions (which is already a generalisation of the corresponding result in \cite{HM1}).

\end{remark}

%\begin{remark}
%\label{rem:oldproblem}
%In \cite[Prop. 2.16]{HM1} and \cite[Prop. 2.9]{HM2}, the sequence from \cite[\S 2]{KKV} or \cite[Thm. 4.2.2]{Brion} was used for the Cox sheaf $\RS(X)$ and $\D{\Cl(X)}$:
%\[
%\begin{tikzcd}[column sep=small, row sep=0]
%1 \ar[r] &  H^0(\stX, \OO^*) \ar[r] &  H^0(\Spec_X \RS(X), \OO^*) \ar[r] &  \Cl(X) \ar[r] & {} \\
%{} \ar[r] &  \Pic(\stX) \ar[r] &  \Pic(\Spec_X \RS(X)) \ar[r] &  \Pic(\D{\Cl(X)} \times \Spec_X \RS(X))
%\end{tikzcd}
%\]
%This application is problematic, as $\Cl(X)$ will not be torsion free in general, hence $\D{\Cl(X)}$ not connected.
%By using \autoref{thm:rs}, we can circumvent the problem there by using the presentation $\Xcan = \stackquot{\Spec_X\SS(X)}{\D{K_0}}$ instead, where $\D{K_0}$ is connected.
%\end{remark}

\addtocontents{toc}{\protect\setcounter{tocdepth}{-1}}
\addtocontents{toc}{\protect\setcounter{tocdepth}{-1}}

\end{document}